\documentclass[12pt]{amsart}
\usepackage{amssymb,amsmath,amsfonts,latexsym,setspace}
\usepackage{bm}
\theoremstyle{plain}
\newtheorem{propn}{Proposition}[section]
\newtheorem{thm}[propn]{Theorem}
\newtheorem{lemma}[propn]{Lemma}
\newtheorem{cor}[propn]{Corollary}
\newtheorem*{thm*}{Theorem}

\theoremstyle{definition}

\theoremstyle{remark}

\newcommand{\Hil}{\mathcal{H}}

\newcommand{\Kil}{\mathcal{K}}

\newcommand{\norm}[1]{\lVert #1 \rVert}

\newcommand{\s}{\mathcal{S}}
\newcommand{\D}{\mathbb{D}}
\newcommand{\W}{\mathcal{W}}

\newcommand{\Nat}{\mathbb{N}}
\newcommand{\Comp}{\mathbb{C}}

\newcommand{\ot}{\otimes}

 \DeclareMathOperator{\Ker}{Ker}

 \DeclareMathOperator{\R}{Re}

\begin{document}
\title[Wandering subspaces of the Bergman space
 and the Dirichlet space over $\mathbb{D}^n$]
 {Wandering subspaces of the Bergman space and the Dirichlet space over polydisc}
\author[Chattopadhyay] {A. Chattopadhyay }

\address{
(A. Chattopadhyay) Indian Statistical Institute \\ Statistics and
Mathematics Unit \\ 8th Mile, Mysore Road \\ Bangalore \\ 560059 \\
India}

\email{arup@isibang.ac.in, 2003arupchattopadhyay@gmail.com}

\author[Das] {B. Krishna Das}

\address{
(B. K. Das) Indian Statistical Institute \\ Statistics and
Mathematics Unit \\ 8th Mile, Mysore Road \\ Bangalore \\ 560059 \\
India}

\email{dasb@isibang.ac.in, bata436@gmail.com}

\author[Sarkar]{Jaydeb Sarkar}

\address{
(J. Sarkar) Indian Statistical Institute \\ Statistics and
Mathematics Unit \\ 8th Mile, Mysore Road \\ Bangalore \\ 560059 \\
India}

\email{jay@isibang.ac.in, jaydeb@gmail.com}

\author[Sarkar]{S. Sarkar}

\address{
(S. Sarkar)  Department of Mathematics\\
  Indian Institute of Science\\
  Bangalore\\
  560 012\\
  India}

\email{santanu@math.iisc.ernet.in}


 \begin{abstract}
 Doubly commutativity of invariant subspaces of the
 Bergman space and the Dirichlet space over the unit polydisc $\mathbb{D}^n$ (with $ n \geq 2$) is investigated.
 We show that for any non-empty subset $\alpha=\{\alpha_1,\dots,\alpha_k\}$
 of $\{1,\dots,n\}$ and doubly commuting invariant subspace $\s$ of
 the Bergman space or the Dirichlet space over $\D^n$, the tuple consists
 of restrictions of co-ordinate multiplication operators
 $M_{\alpha}|_\s:=(M_{z_{\alpha_1}}|_\s,\dots, M_{z_{\alpha_k}}|_\s)$
 always possesses wandering subspace of the form
 \[\bigcap_{i=1}^k(\s\ominus z_{\alpha_i}\s).
  \]
\end{abstract}

\maketitle

\section{Introduction}
A closed subspace $\W$ of a Hilbert space $\Hil$ is said to be
\emph{wandering subspace} (following Halmos \cite{H}) for an
$n$-tuple of commuting bounded linear operators $T=(T_1,\dots,T_n)$
($n \geq 1$) if
  \[\W\perp
 T_{1} ^{l_1} T_{2} ^{l_2}\cdots T_{n} ^{l_n} \W
 \]
 for all $(l_1,\dots,l_n)\in \mathbb{N}^n\setminus \{(0,\dots, 0)\}$
 and \[\mathcal{H} = \overline{\mbox{span}}\{ T_{1} ^{l_1} T_{2} ^{l_2}\cdots T_{n}
 ^{l_n} h : h \in \W, l_1, \ldots, l_n \in \mathbb{N}\}.\]In this case, the tuple $T$ is said to have the \textit{wandering subspace
 property}.

The main purpose of this paper is to investigate the following
question:

\noindent \textsf{Question:} Let $(T_{1}, \ldots, T_{n})$ be a
commuting $n$-tuple of bounded linear operators on a Hilbert space
$\mathcal{H}$. Does there exists a
 wandering subspace $\W$ for $(T_{1}, \ldots, T_{n})$?

This question has an affirmative answer for the restriction of
multiplication operator by the coordinate function $M_z$, to an
invariant subspace of the Hardy space $H^2(\mathbb{D})$ (Beurlings
theorem \cite{AB}) or the Bergman space $A^2(\mathbb{D})$ (Aleman,
Richter and Sundberg \cite{ARS}) or the Dirichlet space
$\mathcal{D}(\mathbb{D})$ (Richter \cite{SR} ) over the unit
disc $\D$ of the complex plane $\mathbb{C}$. For $n \geq 2$,
existence of wandering subspaces for general invariant subspaces
of the Hardy space $H^2(\mathbb{D}^n)$ over the unit polydisc
$\mathbb{D}^n$ rather fails spectacularly (cf. \cite{RU}).

Recall that the Hardy space over the unit polydisc $\mathbb{D}^n =
\{ \bm{z} = (z_1, \ldots, z_n) \in \mathbb{C}^n : |z_i| < 1, i = 1,
\ldots, n\}$ is denoted by $H^2(\D^n)$ and defined by
 $$H^2(\mathbb{D}^n)  = \{  f \in \mathcal{O} (\mathbb{D}^n) :
 \mathop{\mbox{sup}}_{0 \leq r < 1}
\mathop{\int}_{\mathbb{T}^n} |f(r \bm{z})|^2 d\bm{\theta}
 < \infty\},$$where $d\bm{\theta}$ is
the normalized Lebesgue measure on the torus $\mathbb{T}^n$, the
distinguished boundary of $\mathbb{D}^n$, $r\bm{z} : = (rz_1,
\ldots, r z_n)$ and $\mathcal{O}(\D^n)$ denotes the set of all
holomorphic functions on $\D^n$ (cf. \cite{RU}).

In \cite{AB}, A. Beurling characterize all closed $M_z$-invariant
subspaces of $H^2(\D)$ in the following sense: Let $\s \neq \{0\}$
be a closed subspace
 of $H^{2}(\mathbb{D}).$ Then $\s$ is $M_z$-invariant if and only if
 $\s = \theta H^{2}(\mathbb{D})$ for some inner function $\theta$ (that is, $\theta \in
 H^\infty(\mathbb{D})$ and $|\theta| = 1$ a.e. on the unit circle $\mathbb{T}$).
In particular, Beurlings theorem yields the wandering subspace
property for $M_z$-invariant subspaces of $H^2(\mathbb{D})$ as
follows: if $\s = \theta H^2(\mathbb{D})$ is an $M_z$-invariant
subspace of $H^2(\mathbb{D})$ then
\begin{equation}\label{1-S}\s = \mathop{\sum}_{m \geq 0} \oplus z^m
\W,\end{equation}where $\W$ is the wandering subspace for $M_z|_\s$
given by
\[\W = \s \ominus z \s = \theta H^2(\mathbb{D}) \ominus z \theta
H^2(\mathbb{D}) = \theta \mathbb{C}.\]

In \cite{RU}, W. Rudin showed that there are invariant subspaces
$\mathcal{M}$ of $H^2(\mathbb{D}^2)$ which does not contain any
bounded analytic function. In particular, the Beurling like
characterization of $(M_{z_1}, \ldots, M_{z_n})$-invariant subspaces
of $H^2(\mathbb{D}^n)$, in terms of inner functions on
$\mathbb{D}^n$ is not possible.

On the other hand, the wandering subspace (and the Beurlings
theorem) for the shift invariant subspaces of the Hardy space
$H^2(\mathbb{D})$ follows directly from the classical Wold
decomposition \cite{W} theorem for isometries. Indeed, for a closed
$M_z$-invariant subspace $\s (\neq \{0\})$ of $H^2(\mathbb{D})$,
\[\mathop{\bigcap}_{m=0}^\infty (M_z|_{\s})^m \s =
\mathop{\bigcap}_{m=0}^\infty M_z^m \s \subseteq
\mathop{\bigcap}_{m=0}^\infty M_z^m H^2(\mathbb{D}) =
\{0\}.\]Consequently, by Wold decomposition theorem for the isometry
$M_{z}|_\s$ on $\s$ we have \[\s = \mathop{\sum}_{m \geq 0} \oplus
z^m \W \oplus (\mathop{\bigcap}_{m=0}^\infty (M_z|_{\s})^m \s) =
\mathop{\sum}_{m \geq 0} \oplus z^m \W,\]where $\W = \s \ominus z
\s$ and hence, (\ref{1-S}) follows. Therefore, the notion of
wandering subspaces is stronger (as well as of independent interest)
than the Beurlings characterization of $M_z$-invariant subspaces of
$H^2(\mathbb{D})$.

To proceed, we now recall the following definition: a commuting
$n$-tuple ($n \geq 2$) of bounded linear operators $(T_1, \ldots,
T_n)$ on $\mathcal{H}$ is said to be \textit{doubly commuting} if
$T_i T_j^* = T_j^* T_i$ for all $1 \leq i < j \leq n$.

Natural examples of doubly commuting tuple of operators are the
multiplication operators by the coordinate functions acting on the
Hardy space or the Bergman space or the Dirichlet space over the unit
polydisc $\mathbb{D}^n$ ($n \geq 2$).

Let $\mathcal{H} \subseteq \mathcal{O}(\mathbb{D}^n)$ be a
reproducing kernel Hilbert space over $\mathbb{D}^n$ (see \cite{Ar})
and multiplication operators $\{M_{z_1}, \ldots, M_{z_n}\}$ by
 coordinate functions are bounded. Then a closed $(M_{z_1},
\ldots, M_{z_n})$-invariant subspace $\s$ of $\mathcal{H}$ is said
to be doubly commuting if the $n$-tuple $(M_{z_1}|_\s, \ldots,
M_{z_n}|_\s)$ is doubly commuting, that is, $R_iR_j^*=R_j^*R_i$ for
all $1 \leq i < j \leq n$, where $R_i=M_{z_i}|_{\s}$.

In \cite{SSW}, third author and Sasane and Wick proved that any
doubly commuting invariant subspace of $H^2(\mathbb{D}^n)$ (where
$n \geq 2$) has the wandering subspace property (see \cite{M} for $n
= 2$). Also in \cite{RT},  Redett and Tung obtained the same
conclusion for doubly commuting invariant subspaces of the
Bergman space $A^2(\mathbb{D}^2)$ over the bidisc $\mathbb{D}^2$.

In this paper we prove that doubly commuting invariant subspaces
of the Bergman space $A^2(\mathbb{D}^n)$ and the Dirichlet space
$\mathcal{D}(\mathbb{D}^n)$ has the wandering subspace property. Our
result on the Bergman space over polydisc is a generalization of the
base case $n = 2$ in \cite{RT}.
Our analysis is based on the Wold-type
decomposition result of S. Shimorin for operators closed to
isometries \cite{SH}.

The paper is organized as follows: in Section 2, we obtain some
general results concerning the wandering subspaces of tuples of
doubly commuting operators. We obtain wandering
subspaces for doubly commuting shift invariant subspaces of the
Bergman and Dirichlet spaces over $\mathbb{D}^n$ in Section 3.

\section{Wandering subspace for tuple of doubly commuting operators}

In this section we prove the multivariate version of the S.
Shimorin's result for tuple of doubly commuting operators on a
general Hilbert space. We show that for a tuple of doubly
 commuting operators $T=(T_1,\dots,T_n)$ on a Hilbert space $\Hil$,
 if for any reducing subspace $\s_i$ of $T_i$
 the subspace $\s_i\ominus T_i\s_i$ is a wandering subspace for $T_i|_{\s_i}$, $i=1,\dots,n$,
  then $\bigcap\limits_{i=1}^n(\Hil\ominus T_{i}\Hil)$ is a
  wandering subspace for $T$.  We \emph{fix for the
 rest of the paper} a natural number $n\geq 2$ and set $\Lambda_n:=\{1,\dots,n\}$.

For a closed subset $\Kil$ of a Hilbert space $\Hil$, an
$n$-tuple of commuting operators
 $T = (T_1,\dots,T_n)$ on $\Hil$ and a non-empty subset $\alpha=
 \{\alpha_1,\dots,\alpha_k\}\subseteq \Lambda_n$,
 we write $[\Kil]_{T_\alpha}$ to denote the smallest
 closed joint $T_{\alpha}:=(T_{\alpha_1},\dots, T_{\alpha_k})$-invariant
 subspace of $\Hil$ containing $\Kil$. In other words

 \begin{equation}
 [\Kil]_{T_\alpha} = \bigvee _{l_1, l_2,\dots, l_k = 0}^{\infty}
 T_{\alpha_1} ^{l_1} T_{\alpha_2} ^{l_2}\cdots T_{\alpha_k} ^{l_k} (\Kil).
 \label{smallest generating set}
 \end{equation} If $\alpha$ is a singleton set $\{i\}$ then we simply write
$[\Kil]_{T_i}$.

For a non-empty subset
$\alpha=\{\alpha_1,\dots,\alpha_k\}\subseteq\Lambda_n$,
 we also denote by $\mathcal{W}_{\alpha}$ the following subspace of $\Hil$,
 \begin{equation}
  \mathcal{W}_{\alpha}=\bigcap\limits_{i=1}^k(\Hil\ominus
  T_{\alpha_i}\Hil).
  \label{wandering subspace}
 \end{equation}
 Again if $\alpha$
 is a singleton set $\{i\}$ then we simply write $\mathcal{W}_i$.
 Thus with the above notation
 $\mathcal{W}_{\alpha}=\bigcap\limits_{\alpha_i\in\alpha}\mathcal{W}_{\alpha_i}$.

 A bounded linear operator $T$ on a Hilbert space $\Hil$ is \emph{analytic}
 if $\bigcap\limits_{n=0}^\infty T^n\Hil=\{0\}$ and it is \emph{concave} if it satisfies
 the following inequality
 \[
  \norm{T^2x}^2+\norm{x}^2\le 2\norm{Tx}^2\quad (x\in\Hil).
 \]Multiplication by coordinate functions on the Dirichlet space over the unit polydisc are
 concave operators as we show in the next section.

  For a single bounded operator $T$ on a Hilbert space $\Hil$,
  the following result ensure the existence of the wandering subspace
  under certain conditions (see \cite{SR}, \cite{SH}).
  \begin{thm} (Richter, Shimorin). Let $T$ be an analytic operator on
  a Hilbert space $\Hil$ satisfies one of the following
 properties:\\
 \textup{(i)} $\parallel T x + y \parallel ^2 \leq 2
(\parallel x \parallel ^2 + \parallel T y \parallel^2 )\ (x,y \in \Hil),$\\
\textup{(ii)} $T$ is concave.\\
 Then $\Hil\ominus T\Hil$ is a wandering subspace for $T$, that is,
 \[\Hil = [\Hil \ominus T\Hil]_T.\]
 \label{shimorin}
 \end{thm}
 The
 following proposition is essential in order to generalize
 the above result for certain tuples of commuting operators.

 \begin{propn}\label{reducing}
  Let $T=(T_1,\dots,T_n)$ be a doubly commuting tuple of operators
  on $\Hil$. Then
  $\W_{\alpha}$ is $T_j$-reducing subspace for all non-empty subset
  $\alpha\subseteq\Lambda_n$ and $j\notin \alpha$, where $\W_\alpha$
  is as in~\eqref{wandering subspace}.
  \end{propn}

  \begin{proof}
  Let $\alpha=\{\alpha_1,\dots,\alpha_k\}$ be a
  non-empty subset of $\Lambda_n$ and $j\notin \alpha$.
  First note that $\mathcal{W}_{l}=\Ker T_{l}^*$ for all $1\le l\le n$ and therefore
 $\mathcal{W}_{\alpha}=\bigcap\limits_{i=1}^k \Ker T_{\alpha_i}^*$.
 Let $x \in \mathcal{W}_{\alpha}$, $\alpha_i\in \alpha$ and $y\in\Hil$.
 By doubly commutativity of $T$ we have,
  \begin{align*} \langle T_j x , T_{\alpha_i} y\rangle& = \langle
 T_{\alpha_i}^* T_j x ,y \rangle =\langle T_j T_{\alpha_i}^*x , y\rangle = 0.
  \end{align*}
  Therefore, $T_j \mathcal{W}_{\alpha} \perp T_{\alpha_i}\Hil$ and hence $T_j \mathcal{W}_{\alpha}
  \subseteq \mathcal{W}_{\alpha_i}$
  for all $\alpha_i\in\alpha$. Thus $\mathcal{W}_{\alpha}$ is an invariant subspace
  for $T_j$.
  Also, by commutativity of $T$ we have \begin{align*}
  \langle T_j^* x , T_{\alpha_i} y\rangle = \langle T_{\alpha_i}^*T_j^* x ,y \rangle
  =\langle T_j^* T_{\alpha_i}^*x , y\rangle =
  0,
  \end{align*}for all $\alpha_i\in\alpha$ and $y\in\Hil$.
  This implies $T_j^*\mathcal{W}_{\alpha} \perp T_{\alpha_i}\Hil$ and then
  $T_j^*\mathcal{W}_{\alpha}\subseteq \W_{\alpha_i}$ for all $\alpha_i\in\alpha$ .
  This completes the proof.
  \end{proof}
   Now we prove the main theorem in the general Hilbert space operator
  setting. Below for a set $\alpha$, we denote by $\#\alpha$
  the cardinality of $\alpha$.
  \begin{thm}
  \label{main}
Let $T = (T_1,\dots,T_n)$ be a doubly commuting tuple of operators
on $\Hil$ such that for any reducing subspace $\s_i$ of $T_i$, the 
subspace $$\s_i\ominus T_i\s_i$$ is a wandering subspace for $T_i|_{\s_i}$,
$i=1,\dots,n$. Then for each non-empty subset
$\alpha=\{\alpha_1,\dots,\alpha_k\}\subseteq \Lambda_n$, the tuple
$T_{\alpha}=(T_{\alpha_1},\dots,T_{\alpha_k})$ has
  the wandering subspace property. Moreover, the corresponding
  wandering subspace is given by
  \[
   \W_{\alpha}=\bigcap_{i=1}^k (\Hil\ominus T_{\alpha_i}\Hil).
  \]
  \end{thm}
 \begin{proof}
  First note that we only need to show
 $\Hil = [\mathcal{W}_{\alpha}]_{T_{\alpha}}$ for any non-empty subset
 $\alpha$ of $\Lambda_n$ as the orthogonal property for wandering
 subspace is immediate. Now for a singleton set $\alpha$ the result follows form the
 assumption that $\W_i$ is a wandering subspace for $T_i$, $i=1,\dots,n$. Now for $\#\alpha\ge 2$, it suffices to show that
 $[\mathcal{W}_{\alpha}]_{T_{\alpha_i}}=\mathcal{W}_{\alpha\setminus \{\alpha_i\}}$
 for any $\alpha_i\in\alpha$.
 Because for $\alpha_i,\alpha_j\in\alpha$, one can repeat the procedure to get
 $$[\mathcal{W}_{\alpha}]_{T_{\{\alpha_i,\alpha_j\}}}=
 [[\mathcal{W}_{\alpha}]_{T_{\alpha_i}}]_{T_{\alpha_j}}=
  \mathcal{W}_{\alpha\setminus\{\alpha_i,\alpha_j\}},$$
 and continue this process until the set $\alpha\setminus\{\alpha_i,\alpha_j\}$
 becomes a singleton set and finally apply the assumption for singleton set.

  To this end, let $\alpha\subseteq\Lambda_n, \#\alpha\ge 2$
  and $\alpha_i\in\alpha$.
  Consider the set
  $F=\mathcal{W}_{\alpha\setminus\{\alpha_i\}}
  \ominus T_{\alpha_i}(\mathcal{W}_{\alpha\setminus\{\alpha_i\}})$.
  Now by Proposition~\ref{reducing}, $\mathcal{W}_{\alpha\setminus \{\alpha_i\}}$
  is a reducing subspace for
  $T_{\alpha_i}$ and therefore
  \begin{align*}
 F&=\{x\in \mathcal{W}_{\alpha\setminus \{\alpha_i\}}: x\in\Ker T_{\alpha_i}^*\}\\
 &=\W_{\alpha\setminus\{\alpha_i\}}\cap\W_{\alpha_i}\\
  &=\mathcal{W}_{\alpha}.
  \end{align*}
  On the other hand, since $\W_{\alpha\setminus\{\alpha_i\}}$ is a reducing 
  subspace for $T_{\alpha_i}$ then by assumption $F=\W_{\alpha}$ is a wandering 
  subspace for $T_{\alpha_i}$. Thus 
  $$[\W_{\alpha}]_{T_{\alpha_i}}=\W_{\alpha\setminus\{\alpha_i\}}.$$
  This completes the proof.
 \end{proof}
  Combining Theorem~\ref{shimorin} with the above theorem we have the following
  result, which is the case of our main interest.
 \begin{cor}
 \label{main2}
  Let $T=(T_1,\dots,T_n)$ be a commuting tuple of analytic operators on $\Hil$
  such that $T$ is doubly commuting and satisfies one of the following
  properties:\\
   \textup{(a)} $T_i$ is concave for each $i=1,\dots,n$\\
  \textup{(b)} $\norm{T_i x + y}^2 \leq 2
  (\norm{x}^2 + \norm{T_i y}^2 )\
  (x,y \in \Hil, i= 1,2,\dots,n).$\\
  Then for any non-empty subset
 $\alpha=\{\alpha_1,\dots,\alpha_k\}\subseteq \Lambda_n$,
  $\W_{\alpha}$ is a wandering subspace for
  $T_{\alpha}=(T_{\alpha_1},\dots,T_{\alpha_k})$, where $\W_{\alpha}$
 is as in~\eqref{wandering subspace}.
 \end{cor}
 \begin{proof}
  Note that if $T$ satisfies condition (a) or 
  (b) then for any $1\le i\le n$ and reducing subspace $\s_i$ of $T_i$,
  $T_i|_{\s_i}$ also satisfies 
  condition (a) or (b) respectively. Thus by Theorem~\ref{shimorin},
  $T$ satisfies the hypothesis of the above theorem and the proof follows.
 \end{proof}

  We end this section by proving kind of converse of the above result
   and part of which is a generalization of
  \cite{RT}, Theorem 3.
 \begin{thm}
 \label{main1}
 Let $T = (T_1,\dots,T_n)$ be an $n$-tuple of commuting
 operators on $\Hil$ with the property
 \[\norm{T_i x + y}^2 \leq 2
  (\norm{x}^2 + \norm{T_i y}^2 )\
  (x,y \in \Hil, i= 1,2,\dots,n)
  \]
  or $T_i$ is concave for all $i=1,\dots,n$.
  Then,\\
  \textup{ (i)} $T$ is doubly commuting on $\Hil$, and\\
  \textup{(ii)}  $T_i$ is analytic for all
  $i= 1,2,\dots, n$\\
  if and only if \\
 \textup{(a)} for any non-empty subset
 $\alpha=\{\alpha_1,\dots,\alpha_k\}\subseteq \Lambda_n$, $\W_{\alpha}$
 is a wandering subspace for $T_{\alpha}=(T_{\alpha_1},\dots,T_{\alpha_k})$
  and for $k\ge 2,\ [\mathcal{W}_{\alpha}]_{T_{\alpha_i}}=
  \mathcal{W}_{\alpha\setminus\{\alpha_i\}}$ for all $\alpha_i\in\alpha$,\\
 \textup{(b)} $T_i$ commutes with $T_j^* T_j$ for all $1\leq i < j \leq n.$
  \end{thm}

 \begin{proof}
 The forward direction follows from Theorem~\ref{main}.

  For the converse, suppose that (a) and (b) hold. To show (i), let
  $1\le i\le n$ be fixed. By assumption (a),
  $\mathcal{W}_i=[\mathcal{W}_{\{i,j\}}]_{T_j}$ for all $1\le i\neq j\le n$.
   This shows that $\mathcal{W}_i$ is $T_j$ invariant subspace for all $1\le j\neq i\le n$.
 Let $x \in \Hil.$
 Since $[\W_i]_{T_i}=\Hil$ then there exists
  a sequence $x_k$ converging to $x$ such that $$x_k = \sum_{m=0}^{N_k} T_i^m x_{m,k},$$
  where each $N_k$ is a nonnegative integer and $x_{m,k}$ is a member of $\mathcal{W}_i.$
 Now for any $1\le j\neq i\le n$ we have,
 \begin{equation*}
 T_i^{*} T_j x_k =\sum_{m=0}^{N_k} T_i^{*}T_j T_i^{m} x_{m,k}
  =  \sum_{m=1}^{N_k} T_i^{*}T_j T_i^{m} x_{m,k}.
 \end{equation*}
 On the other hand,
 \begin{align*}
 T_j T_i^* x_k = \sum_{m=0}^{N_k} T_j T_i^* T_i^m x_{m,k} = \sum_{m=1}^{N_k}  T_j T_i^* T_i^m x_{m,k}
= \sum_{m=1}^{N_k} T_i^* T_j T_i^m x_{m,k},
 \end{align*}where the last equality follows from (b). So $T_i^* T_j x_k = T_j T_i^*x_k$ and by taking limit
 $T_i^* T_j x = T_j T_i^*x$. Thus we have (i).

\noindent Finally, by Theorem $3.6$ of~\cite{SH} we have that
  $$\Hil = \left[\mathcal{W}_i\right]_{T_i} \oplus \bigcap_{m=1}^\infty T_i^m\Hil,$$
  for all $1\le i\le n$. By part (a), $[\mathcal{W}_i]_{T_i}=\Hil$ for all
  $1\le i\le n$. Thus $\bigcap\limits_{m=1}^\infty T_i^m\Hil=\{0\}$
  for all $1\le i\le n$ and this completes the proof.
 \end{proof}
 \section{Wandering subspaces for invariant subspaces of $A^2(\D^n)$ and $\mathcal{D}(\D^n)$ }
 In this section we apply the general theorem proved in the previous
 section to obtain wandering subspaces for invariant subspaces of
 the Bergman space and the Dirichlet space over polydisc.

 The Bergman space over $\D^n$ is denoted by $A^2(\D^n)$
  and defined by
 $$A^2 (\mathbb{D}^n) =
 \{ f \in \mathcal{O}(\mathbb{D}^n): \int_{\D^n}\vert f(\mathbf{z})\vert^2
 d\mathbf{A}(\mathbf{z})<\infty\},$$ \\
 where $d\mathbf{A}$ is the product area measure on $\D^n$.

  Below for any invariant subspace $\s$ of $A^2(\D^n)$ and non-empty set
  $\alpha=\{\alpha_1,\dots,\alpha_k\}\subseteq\Lambda_n$,
  we denote by $\mathcal{W}_{\alpha}^{\s}$ the following set:
  \[\mathcal{W}^\s_{\alpha}:=\bigcap_{i=1}^k(\s\ominus z_{\alpha_i}\s).\]
  We denote by $M=(M_{z_1},\dots,M_{z_n})$ the $n$-tuple of co-ordinate
  multiplication operators on $A^2(\D^n)$.

 \begin{thm}
 Suppose $\s$ is a closed joint $M$-invariant subspace
 of $A^2(\mathbb{D}^n)$.
 If $\s$ is doubly commuting then
  for any non-empty
 subset $\alpha=\{\alpha_1,\dots,\alpha_k\}$ of $\Lambda_n$, $\W^\s_{\alpha}$
 is a wandering subspace for
 $M_{\alpha}|_\s=(M_{z_{\alpha_1}}|_\s,\dots,M_{z_{\alpha_k}}|_\s)$.
\end{thm}
 The proof of the above theorem follows if we show the tuple of operators
 $$(M_{z_1}|_{\s}\dots, M_{z_n}|_{\s})$$
 on $\s$ satisfies all the hypothesis of Corollary~\ref{main2}.
 First note that by analyticity of $A^2(\D^n)$, all the co-ordinate
 multiplication operators $M_{z_i}, i=1,\dots,n$ on $A^2(\D^n)$ are analytic. Then
 $M_{z_i}|_{\s}$ is also analytic for all $i=1,\dots,n$.
 Thus the only hypothesis remains to verify is either
 condition (a) or (b) of Corollary~\ref{main2}. In the
 next lemma we show that the tuple $(M_{z_1},\dots,M_{z_n})$ satisfies condition
 (b) and therefore so does $(M_{z_1}|_{\s}\dots, M_{z_n}|_{\s})$.
 \begin{lemma}
 For all $1\leq i \leq n$ the operators
 $$M_{z_i} : A^2 (\mathbb{D}^n) \rightarrow
A^2 (\mathbb{D}^n),\quad f\mapsto
z_i f,$$ satisfies $$\|M_{z_i} f + g\|^2
 _{A^2 (\mathbb{D}^n)}  \leq  2 \left(\|f\|^2_{A^2 (\mathbb{D}^n)}
 + \|M_{z_i} g\|^2 _{A^2 (\mathbb{D}^n)}\right),$$ for all
 $f,g \in A^2 (\mathbb{D}^n).$
\end{lemma}

 Before we prove this lemma, we recall a well known fact regarding the norm
 of a function in $A^2 (\mathbb{D}^n)$ and prove an inequality.
 If $f$ is in $A^2 (\mathbb{D}^n)$
 with the following power series expansion corresponding to $i$-th variable:
 $$f(z_1,\cdots,z_n) = \sum\limits_{k = 0} ^\infty \mathbf{a}_k
 z_i^{k},$$
  where $\mathbf{a}_k\in A^2(\D^{n-1})$ for all $k\in\Nat$,
  then $$ \| f \| ^2 _{A^2 (\mathbb{D}^n)}
 =  \sum\limits_{k = 0} ^\infty  \frac{\norm{\mathbf{a}_{k}}_{A^2(\D^{n-1})}^2 }{(k +1)}.$$

  Next we prove the following inequality (can be found in ~\cite{DS}, page 277,
  we include the proof for completeness) for any $z,w\in\Comp$ and $k\in\Nat\setminus\{0\}$,
  \begin{equation}
   \label{inequality}
   \frac{|z+w|^2}{k+1}\le 2\left(\frac{|z|^2}{k}+\frac{|w|^2}{k+2}\right).
  \end{equation}
  To this end, note that $2k(k+2)\R(z\bar{w})\le (k+2)^2|z|^2+k^2|w|^2$, which follows from
  the inequality $|(k+2)z-kw|^2\ge 0$. Now
  \begin{align*}
   \frac{|z+w|^2}{k+1}&= \frac{|z|^2+|w|^2+2\R(z\bar{w})}{k+1}\\
   &\le  \frac{|z|^2+|w|^2}{k+1}+ \frac{|z|^2(k+2)/k+|w|^2k/(k+2)}{k+1}\\
   &=2\left(\frac{|z|^2}{k}+\frac{|w|^2}{k+2}\right).
  \end{align*}

 Now we prove the lemma.
 \begin{proof}
 Let $1\le i\le n$ be fixed. Let
  $f(z_1,\cdots,z_n) = \sum\limits_{k = 0} ^\infty
  \mathbf{a}_{k}
 z_i^{k} $  and
 $g(z_1,\cdots,z_n) = \sum\limits_{k = 0} ^\infty
  \mathbf{b}_{k}
 z_i^{k}$ be the power series expansions of two functions
 $f,g\in A^2(\D^n)$ with respect to $z_i$-th variable, where
 $\mathbf{a}_k, \mathbf{b}_k\in A^2(\D^{n-1})$ for all $k\in\Nat$.
  Then
 \begin{align*}
  (M_{z_i}f+g)=\sum_{k= 0} ^\infty
  \mathbf{a}_{k}
 z_i^{k+1 }+ \sum_{k=0}^{\infty}\mathbf{b}_k z_i^{k}
 = \mathbf{b}_0+ \sum_{k=1} ^\infty (\mathbf{a}_{k-1}+\mathbf{b}_k) z_i^k.
 \end{align*}
  Now
  \begin{align*}
 \|M_{z_i}f+g \| ^2
  _{A^2 (\mathbb{D}^n)} &= \|\mathbf{b}_0\|^2_{A^2(\D^{n-1})}+\sum_{k = 1} ^\infty
  \frac{\|\mathbf{a}_{k-1}+\mathbf{b}_k\|_{A^2(\D^{n-1})}^2}{(k+1)}\\
  &\le  \|\mathbf{b}_0\|^2_{A^2(\D^{n-1})}+ 2\sum_{k=1}^\infty \left(
  \frac{\|\mathbf{a}_{k-1}\|^2_{A^2(\D^{n-1})}}{k}+
  \frac{ \|\mathbf{b}_k\|^2_{A^2(\D^{n-1})}}{k+2}\right)\ (\text{ by}~\eqref{inequality})\\
  &= 2 \left(\sum_{k=0}^\infty  \frac{\|\mathbf{a}_k\|^2_{A^2(\D^{n-1})}}{k+1} +
  \sum_{k=1}^\infty \frac{\|\mathbf{b}_{k-1}\|^2_{A^2(\D^{n-1})}}{k+1}\right)\\
  &= 2 \left(\|f\|^2_{A^2 (\mathbb{D}^n)}
 + \|M_{z_i} g\|^2 _{A^2 (\mathbb{D}^n)}\right).
  \end{align*}
   Thus the result.
   \end{proof}
 Now we turn our discussion to the Dirichlet space over polydisc.
 The Dirichlet space over $\D$ is denoted by $\mathcal{D}(\D)$ and defined
 by
 \[
  \mathcal{D}(\D):=\{f\in \mathcal{O}(\D): f'\in A^2(\D)\}.
 \]
 For any $f=\sum\limits_{k=0}^\infty a_kz^k\in\mathcal{D}(\D)$, $\norm{f}_{\mathcal{D}}^2=
 \sum\limits_{k=0}^\infty (k+1)\vert a_k\vert^2$. The Dirichlet space over
 $\D^n$ is denoted by $\mathcal{D}(\D^n)$ and defined by
 $$\mathcal{D}(\D^n):=\underbrace{\mathcal{D}(\D)\ot\cdots\ot\mathcal{D}(\D)}
 \limits_{\text{n-times}}.$$
 Another way of expressing $\mathcal{D}(\D^n)$ is the following
 \[
  \mathcal{D}(\D^n):=\{f\in \mathcal{O}\big(\D; \mathcal{D}(\D^{n-1})\big):
  f=\sum_{k=0}^\infty \mathbf{a}_k z^k,
  \sum_{k=0}^\infty (k+1)\norm{\mathbf{a}_k}
  _{\mathcal{D}(\D^{n-1})}^2
  <\infty\}.
 \]
 In the above, $\mathbf{a}_k\in\mathcal{D}(\D^{n-1})$ for all $k\in\Nat$.
 The co-ordinate multiplication operators on $\mathcal{D}(\D^n)$ are also denoted by
 $M_{z_i}$, $i=1,\dots,n$. Since $\mathcal{D}(\D^n)$ contains
 holomorphic functions on $\D^n$ then $M_{z_i}$ is analytic for
 all $i=1,\dots,n$. Let $1\le i\le n$ be fixed. Now for $f\in \mathcal{D}(\D^n)$,
 let $f=\sum\limits_{k=0}^\infty \mathbf{a}_k z_i^k$ be the Taylor expansion
 of $f$ corresponding to $z_i$-th variable, where $\mathbf{a}_k\in \mathcal{D}(\D^{n-1})$
 for all $k\in\Nat$. Then $\norm{f}^2=
 \sum\limits_{k=0}^\infty(k+1)\norm{\mathbf{a}_k}_{\mathcal{D}(\D^{n-1})}^2$
 and
 \begin{align*}
  \norm{M_{z_i}^2f}^2+\norm{f}^2&=
  \sum_{k=2}^\infty (k+1)\norm{\mathbf{a}_{k-2}}_{\mathcal{D}(\D^{n-1})}^2
  +\sum_{k=0}^\infty(k+1)\norm{\mathbf{a}_k}_{\mathcal{D}(\D^{n-1})}^2\\
  &=\sum_{k=0}^{\infty}(k-1)\norm{\mathbf{a}_k}_{\mathcal{D}(\D^{n-1})}^2
  +\sum_{k=0}^\infty(k+1)\norm{\mathbf{a}_k}_{\mathcal{D}(\D^{n-1})}^2\\
  &=2\sum_{k=1}^{\infty}k\norm{\mathbf{a}_k}_{\mathcal{D}(\D^{n-1})}^2\\
  &=2\norm{M_{z_i}f}^2.
 \end{align*}
 Therefore $M_{z_i}$ is concave for all $i=1,\dots,n$. Thus
 again by Corollary~\ref{main2}, we have proved the following result
 of wandering subspaces for invariant subspaces of Dirichlet space
 over polydisc.
 \begin{thm}
 Suppose $\s$ is a closed joint $(M_{z_1},\dots,M_{z_n})$-invariant subspace
 of $\mathcal{D}(\D^n)$.
 If $\s$ is doubly commuting then
  for any non-empty
 subset $\alpha=\{\alpha_1,\dots,\alpha_k\}$ of $\Lambda_n$, $\W^\s_{\alpha}$
 is a wandering subspace for
 $M_{\alpha}|_\s=(M_{z_{\alpha_1}}|_\s,\dots,M_{z_{\alpha_k}}|_\s)$,
  where
  \[
   \W^\s_{\alpha}=\bigcap_{i=1}^k(\s\ominus z_{\alpha_i}\s).
  \]

 \end{thm}

We conclude the paper with the remark that since the hypothesis of
Corollary~\ref{main2} and Theorem~\ref{main1} are same then the same
 conclusion as in Theorem~\ref{main1} holds for invariant subspaces
 of Bergman space or Dirichlet space over polydisc.

\vspace{0.2in}

 \noindent\textbf{Acknowledgment:} First two authors are grateful to Indian
  Statistical Institute, Bangalore Centre for warm hospitality.
  The fourth author was supported by UGC Centre for Advanced Study.


\vspace{.2in}

\begin{thebibliography}{ABC}

 \bibitem [Aro] {Ar}
 N. Aronszajn, \emph{Theory of reproducing kernels}, Trans. of
 American Mathematical Soc. \textbf{68} (1950), 337–-404.

\bibitem[ARS] {ARS}
 A. Aleman,  S. Richter and C. Sundberg,
\emph{ Beurling's theorem for the Bergman space,}
 Acta. Math.
 \textbf{177} (1996),  275--310.

\bibitem [Beu] {AB}
A. Beurling,
\emph{On two problems concerning linear transformations in Hilbert space,}
Acta. Math.
\textbf{81} (1949), 239--255


\bibitem [DuS] {DS}
P. Duren and A. Schuster,
 \emph{Bergman Spaces,}
Math. Surveys and Monographs,
100,   Amer. Math. Soc. Providence.

\bibitem[Hal] {H} P. R. Halmos, \emph{Shifts on Hilbert spaces}, J.
Reine Angew. Math. \textbf{208} (1961) 102--112.

\bibitem [Man] {M}
V. Mandrekar,
\emph{The validity of Beurling theorems in polydiscs,}
Proc. Amer. Math. Soc.
\textbf{103} (1988), 145--148.

\bibitem [ReT] {RT}
D. Redett and J. Tung,
\emph{Invariant subspaces in Bergman space over the bidisc }
Proc.  Amer Math Soc,
\textbf{138} (2010) , 2425--2430.

\bibitem[Ric]{SR}
S. Richter, \emph{Invariant subspaces of the Dirichlet shift}, J.
Reine Angew. Math. \textbf{386} (1988), 205--220.

\bibitem [Rud] {RU}
W. Rudin,
\emph{Function theory in polydiscs,}
Benjamin, New York, 1969.

\bibitem [Shi] {SH}
S. Shimorin,
\emph{Wold-type decompositions and wandering subspaces for operators
close to isometries,}
J. Reine Angew. Math.
\textbf{531}  (2001), 147--189.

\bibitem[SSW]{SSW}
J. Sarkar, A. Sasane and B. Wick, \emph{Doubly commuting submodules
of the Hardy module over polydiscs},  preprint, arXiv:1302.5312.



 \bibitem [Wol] {W}
 H. Wold, \emph{A study in the analysis of stationary time series},
 Almquist and Wiksell, Uppsala, 1938.
\end{thebibliography}
\end{document}